\documentclass{amsart}
\usepackage[mathscr]{eucal}
\usepackage{amssymb}
\usepackage{latexsym}
\usepackage{amsthm}

\theoremstyle{plain}
\newtheorem{Def}{Definition}[section]
\newtheorem{Thm}[Def]{Theorem}
\newtheorem{Prop}[Def]{Proposition}
\newtheorem{Rem}[Def]{Remark}
\newtheorem{Ex}[Def]{Example}
\newtheorem{Cor}[Def]{Corollary}

\newtheorem{Lem}[Def]{Lemma}

\numberwithin{equation}{section}

\title{
Note on mod $p$ property of Hermitian modular forms
}
\author{Toshiyuki Kikuta and Shoyu Nagaoka}

\begin{document}
\maketitle
\footnote{New title :
On the theta operator for Hermitian modular
forms of degree 2}
%
\noindent
%
\begin{abstract}
The mod $p$ kernel of the theta operator is the set of modular
forms whose image of the theta operator is congruent to zero
modulo a prime $p$. In the case of Siegel modular forms, the
authors found interesting examples of such modular forms.
For example, Igusa's odd weight cusp form is an element of 
 mod 23 kernel of the
theta operator. In this paper, we give some examples which represent
elements in the mod $p$ kernel of the theta operator in the case of
Hermitian modular forms of degree 2.
\end{abstract}

\section{Introduction}
\label{intro}
Serre \cite{Serre} developed the theory of $p$-adic and mod $p$ modular
forms and  produced several interesting results.
In his theory, the Ramanujan operator
$\theta :f=\sum a_nq^n \longmapsto \theta (f):=\sum n\,a_nq^n$
played an important role. The notion of such operator was extended
to the case of Siegel modular forms. The theta operator (generalized
Ramanujan operator) on Siegel modular forms is defined by
$$
\varTheta : F=\sum_T a(T)q^T \longmapsto 
\varTheta (F):=\sum_T\text{det}(T)\cdot a(T)q^T,
$$
where $F=\sum a(T)q^T$ is the Fourier expansion (generalized $q$-expansion)
of $F$.

For a prime number $p$, the theta operator acts on the algebra of mod $p$
Siegel modular forms (cf. B\"{o}cherer-Nagaoka \cite{B-N}). In our study, we found Siegel
modular forms $F$ which satisfy the property
$$
\varTheta (F) \equiv 0 \pmod{p}.
$$
The space consisting of such Siegel modular forms is called the {\it mod $p$
kernel of the theta operator}. In this terminology, we can say that the Igusa cusp form
of weight 35 is an element of the the mod 23 kernel of the 
theta operator (cf. Kikuta-Kodama-Nagaoka \cite{K-K-N}). Moreover the theta series attached to 
the Leech lattice is also in the mod 23 kernel of the theta operator 
(cf. Nagaoka-Takemori \cite{N-T}).

The main purpose of this paper is to extend the notion of the theta operator
to the case of Hermitian modular forms and, to give some examples of Hermitian 
modular forms which are in the mod $p$ kernel of the theta operator.
The first half concerns the Eisenstein series. Let 
$\Gamma^2(\mathcal{O}_{\boldsymbol{K}})$ be the Hermitian modular group
of degree 2 with respect to an imaginary quadratic number field $\boldsymbol{K}$.
Krieg \cite{Krieg} constructed a weight $k$ Hermitian modular form 
$F_{k,\boldsymbol{K}}$ which coincides with the weight $k$ Eisenstein series
$E_{k,\boldsymbol{K}}^{(2)}$ for $\Gamma^2(\mathcal{O}_{\boldsymbol{K}})$.

The first result says that the Hermitian modular form $F_{p+1,\boldsymbol{K}}$
is in the mod $p$ kernel of the theta operator under some condition
on $p$, namely
$$
\varTheta (F_{p+1,\boldsymbol{K}}) \equiv 0\pmod{p}
\quad (\text{cf. Theorem \ref{main1}}).
$$
As a corollary, we can show that the weight $p+1$ Hermitian Eisenstein
series $E_{p+1,\boldsymbol{K}}^{(2)}$ satisfies
$$
\varTheta( E_{p+1,\boldsymbol{K}}^{(2)}) \equiv 0 \pmod{p}
$$
if the class number of $\boldsymbol{K}$ equals one.

In the remaining part, we give various examples which are in the mod $p$
kernel of the theta operator. The first example we show is related to the
theta series attached to positive definite, even unimodular Hermitian lattice 
over the Gaussian field.
Let $\mathcal{L}$ be a positive definite, even unimodular Hermitian lattice of rank
$r$ with the Gram matrix $H$. We denote the corresponding Hermitian theta series
of degree $n$ by $\vartheta_{\mathcal{L}}^{(n)}=\vartheta_H^{(n)}$. It is known
that the rank $r$ is divisible by 4 and $\vartheta_{\mathcal{L}}^{(n)}$ becomes a
Hermitian modular form of weight $r$. 
In the case $r=12$, we have a positive definite, even integral Hermitian lattice
$\mathcal{L}_{\mathbb{C}}$ of rank 12, which does not have any vector of
length one. In this paper we call it the Hermitian Leech lattice. 
The theta series attached to $\mathcal{L}_{\mathbb{C}}$ satisfies
$$
\varTheta (\vartheta_{\mathcal{L}_{\mathbb{C}}}^{(2)}) \equiv
0 \pmod{11}
\quad (\text{cf. Theorem \ref{mod11cong}}).
$$

The next example is connected with the Hermitian theta constant. 
Let $\mathcal{E}$ be a set of mod 2 even characteristics of degree 2 (cf. $\S$ \ref{thetaconstant}).
We consider the theta constant $\theta_{\boldsymbol{m}}$\,$(\boldsymbol{m}\in\mathcal{E})$.
It is known that the function
$$
\psi_{4k}:=\frac{1}{4}\sum_{\boldsymbol{m}\in\mathcal{E}}\theta_{\boldsymbol{m}}^{4k}
$$
defines a Hermitian modular form of weight $4k$ (cf. Freitag \cite{Freitag}).
The final result can be stated as
$$
\varTheta (\psi_8) \equiv 0 \pmod{7},\qquad
\varTheta (\psi_{12}) \equiv 0 \pmod{11}
\quad (\text{cf. Theorem \ref{thetaconstantth}}).
$$

Our proof is based on the fact that the image of a weight $k$ modular form of the theta
operator is congruent to a weight $k+p+1$ cusp form mod $p$ (cf. Theorem 2.7) and then
we use the Sturm bound (Corollary 2.6).
\section{Hermitian modular forms}
\label{Sect.2}
\subsection{Notation and definition}
\label{Sect.2.1}
The Hermitian upper half-space of degree $n$ is defined by
$$
\mathbb{H}_n:=\{\,Z\in Mat_n(\mathbb{C})\,\mid\, \frac{1}{2i}(Z-{}^tZ)>0\,\},
$$
where ${}^t\overline{Z}$ is the transposed complex conjugate of $Z$.
The space $\mathbb{H}_n$ contains the Siegel upper-half space of degree
$n$
$$
\mathbb{S}_n:=\mathbb{H}_n\cap {\rm Sym}_n(\mathbb{C}).
$$
Let $\boldsymbol{K}$ be an imaginary quadratic number field with discriminant
$d_{\boldsymbol{K}}$ and ring of integers $\mathcal{O}_{\boldsymbol{K}}$.
The Hermitian modular group
$$
\Gamma^n(\mathcal{O}_{\boldsymbol{K}}):=\left\{\, M\in
{\rm Mat}_{2n}(\mathcal{O}_{\boldsymbol{K}})
\,\mid\,
{}^t\overline{M}J_nM=J_n:=
\begin{pmatrix}0 & -1_n \\ 1_n & 0\end{pmatrix}\,\right\}
$$
acts on $\mathbb{H}_n$ by fractional transformation
$$
\mathbb{H}_n\ni Z\longmapsto M\langle Z\rangle:=
(AZ+B)(CZ+D)^{-1},\;
M=\begin{pmatrix}A&B\\ C&D\end{pmatrix}
\in \Gamma^n(\mathcal{O}_{\boldsymbol{K}}).
$$

Let $\Gamma \subset \Gamma^n(\mathcal{O}_{\boldsymbol{K}})$ be a
subgroup of finite index and $\nu_k$\,$(k\in\mathbb{Z})$ an abelian character
of $\Gamma$ satisfying $\nu_k\cdot\nu_{k'=}\nu_{k+k'}$.
We denote by $M_k(\Gamma,\nu_k)$ the space of Hermitian modular forms
of weight $k$ and character $\nu_k$ with respect to $\Gamma$. Namely
it consists of holomorphic functions $F:\mathbb{H}_n\longrightarrow\mathbb{C}$
satisfying
$$
F\mid_kM(Z):={\rm det}(CZ+D)^{-k}F(M\langle Z\rangle)=\nu_k(M)\cdot F(Z)
$$
for all $M=\binom{*\,*}{CD}\in \Gamma$. When
$\nu_k$ is trivial, we write it by $M_k(\Gamma)$ simply. The subspace
$S_k(\Gamma,\nu_k)$ of cusp forms is characterized by the condition
$$
\Phi\left(F\mid_k\begin{pmatrix}{}^t\overline{U}&0\\ 0& U\end{pmatrix}\right)
\equiv 0,\quad
(U\in GL_n(\mathcal{O}_{\boldsymbol{K}})),
$$
where $\Phi$ is the Siegel operator. A modular form $F\in M_k(\Gamma,\nu_k)$
is called {\it symmetric} if
$$
F({}^t\!Z)=F(Z).
$$
We denote by $M_k(\Gamma,\nu_k)^{\rm sym}$ the subspace consisting of symmetric
modular forms. Moreover
$$
S_k(\Gamma,\nu)^{\rm sym}:=M_k(\Gamma,\nu_k)^{\rm sym}\cap
S_k(\Gamma,\nu_k).
$$
If $F\in M_k(\Gamma,\nu_k)$ satisfies the condition
$$
F(Z+B)=F(Z)\qquad {\rm for\; all} \;B\in Her_n(\mathcal{O}_{\boldsymbol{K}}),
$$
then $F$ has a Fourier expansion of the form
\begin{equation}
\label{Fourier}
F(Z)
=\sum_{0\leq H\in\Lambda_n(\boldsymbol{K})}
a(F;H)\text{exp}(2\pi i\text{tr}(HZ)),
\end{equation}
where
$$
\Lambda_n(\boldsymbol{K}):=\{\,H=(h_{jl})\in Her_n(\boldsymbol{K})\,\mid\,
h_{jj}\in\mathbb{Z},\;\sqrt{d_{\boldsymbol{K}}}\,h_{jl}\in\mathcal{O}_{\boldsymbol{K}}\,\}.
$$

We assume that any $F\in M_k(\Gamma,\nu_k)$ has the Fourier expansion above.
For any subring $R\subset\mathbb{C}$, we write as
$$
M_k(\Gamma,\nu_k)_R:=\{\,F\in M_k(\Gamma,\nu_k)\,\mid\,
a(F;H)\in R\;\;(\forall H\in\Lambda_n(\boldsymbol{K}))\,\}.
$$
In the Fourier expansion (\ref{Fourier}), we use the abbreviation
$$
\boldsymbol{q}^H:=\text{exp}(2\pi i\text{tr}(HZ)).
$$
The generalized $\boldsymbol{q}$-expansion $F=\sum a(F;H)\boldsymbol{q}^H$
can be considered as an element in a formal power series ring
$\mathbb{C}[\![\boldsymbol{q}]\!]$ (cf. Munemoto-Nagaoka \cite{M-N}, p.248) from which we have
$$
M_k(\Gamma,\nu_k)_R \subset R[\![\boldsymbol{q}]\!].
$$

Let $p$ be a prime number and $\mathbb{Z}_{(p)}$ the local ring at $p$, namely, 
ring of $p$-integral rational numbers. For 
$F_i\in M_{k_i}(\Gamma,\nu_{k_i})_{\mathbb{Z}_{(p)}}$\,
($i=1,2$), we write $F_1 \equiv F_2 \pmod{p}$ when
$$
a(F_1;H) \equiv a(F_2;H) \pmod{p}
$$
for all $H\in\Lambda_n(\boldsymbol{K})$.
\subsection{Hermitian modular forms of degree 2}
\label{Sect.2.2}
In the rest of this paper, we deal with Hermitian modular forms of degree 2.
\subsubsection{Eisenstein series}
\label{Sect.2.2.1}
We consider the Hermitian Eisenstein series of degree 2.
$$
E_{k,\boldsymbol{K}}^{(2)}(Z)
:=\sum_{M=\binom{*\,*}{CD}}\text{det}^{k/2}(M)\,\text{det}(CZ+D)^{-k},
\quad Z\in \mathbb{H}_2,
$$
where $k>4$ is even and $M=\binom{*\,*}{CD}$ runs over a set
of representaives of 
$\left\{\binom{*\,*}{0\,*}\right\}\backslash \Gamma^2(\mathcal{O}_{\boldsymbol{K}})$.
Then
$$
E_{k,\boldsymbol{K}}^{(2)}\in
M_k(\Gamma^2(\mathcal{O}_{\boldsymbol{K}}),\text{det}^{-k/2})_{\mathbb{Q}}^{\text{sym}}.
$$
Moreover, $E_{4,\boldsymbol{K}}^{(2)}\in
M_4(\Gamma^2(\mathcal{O}_{\boldsymbol{K}}),\text{det}^{-2})_{\mathbb{Q}}^{\text{sym}}$ is constructed by the Maass
lift (Krieg \cite{Krieg}). For an even integer $k\geq 4$, $E_k^{(1)}=\Phi (E_{k,\boldsymbol{K}}^{(2)})$
is the normalized Eisenstein series Eisenstein series of weight $k$ for $SL_2(\mathbb{Z})$.
\subsubsection{Structure of the graded ring in the case 
$\boldsymbol{K}=\mathbb{Q}(i)$}
\label{Sect.2.2.2}
In this section, we assume that $\boldsymbol{K}=\mathbb{Q}(i)$. In \cite{K-N},
the authors defined some Hermitian cusp forms
$$
\chi_8\in S_8(\Gamma^2(\mathcal{O}_{\boldsymbol{K}}))_{\mathbb{Z}}^{\text{sym}},
F_{10}\in S_{10}(\Gamma^2(\mathcal{O}_{\boldsymbol{K}}),\text{det}^5)_{\mathbb{Z}}^{\text{sym}},
F_{12}\in S_{12}(\Gamma^2(\mathcal{O}_{\boldsymbol{K}}))_{\mathbb{Z}}^{\text{sym}},
$$
characterized by
$$
\chi_8\mid_{\mathbb{S}_2}\equiv 0,\quad
F_{10}\mid_{\mathbb{S}_2}=6X_{10},\quad
F_{12}\mid_{\mathbb{S}_2}=X_{12},
$$
where $X_k$\,($k=10,12$) is Igusa's Siegel cusp form of weight $k$ 
with integral Fourier coefficients (cf. Kikuta-Nagaoka \cite{K-N}).
\begin{Thm}
\label{structure}
Let $\boldsymbol{K}=\mathbb{Q}(i)$. The graded ring
$$
\bigoplus_{k\in\mathbb{Z}}
M_k(\Gamma^2(\mathcal{O}_{\boldsymbol{K}}),\text{det}^{k/2})^{\text{sym}}.
$$
is generated by
$$
E_{4,\boldsymbol{K}}^{(2)},\quad
E_{6,\boldsymbol{K}}^{(2)},\quad
\chi_8,\quad
F_{10}, \;\;\text{and}\;\;
F_{12}.
$$
\end{Thm}
For the proof, we should consult, for example, Kikuta-Nagaoka \cite{K-N}.
\begin{Rem}
$E_{4,\boldsymbol{K}}^{(2)}\in M_4(\Gamma^2(\mathcal{O}_{\boldsymbol{K}}))_{\mathbb{Z}}^{\text{sym}} $,
$E_{6,\boldsymbol{K}}^{(2)}\in M_6(\Gamma^2(\mathcal{O}_{\boldsymbol{K}}),\text{det}^{3})_{\mathbb{Z}}^{\text{sym}} $.
\end{Rem}
\subsubsection{Sturm bound}
\label{Sect.2.2.3}
Sturm gave some condition of $p$-divisibility of the Fourier coefficients of modular form. 
Later the bound is studied in the case of modular forms with several variables.
For example, the first author \cite{C-C-K} studied the bound for Hermitian modular
forms of degree 2 with respect to 
$\boldsymbol{K}=\mathbb{Q}(i)$ and $\mathbb{Q}(\sqrt{3}\,i)$. However the statement
is incorrect. We correct it here.

We assume that $\boldsymbol{K}=\mathbb{Q}(i)$ and use an abbreviation
\begin{equation}
\label{abb}
[m,a+bi,n]:=\begin{pmatrix}m & \frac{a+bi}{2}\\ \frac{a-bi}{2} & n \end{pmatrix}
\in\Lambda_2(\boldsymbol{K}).
\end{equation}
We define a lexicographic order for the different element elements
$$
H=[m,a+bi,n],\quad H'=[m',a'+b'i,n']
$$
of $\Lambda_2(\boldsymbol{K})$ by
\begin{align*}
H \succ H'\quad\Longleftrightarrow\quad & (1)\; \text{tr}(H) > \text{tr}(H')\quad \text{or}\\
                                      & (2)\; \text{tr}(H)=\text{tr}(H'),\;m>m'\quad \text{or}\\
                                     & (3)\;  \text{tr}(H)=\text{tr}(H'),\;m=m',\;a>a'\quad \text{or}\\
                               & (4)\; \text{tr}(H)=\text{tr}(H'),\;m=m',\;a=a',\;b>b'.
\end{align*}
Let $p$ be a prime number and 
$F\in M_k(\Gamma^2(\mathcal{O}_{\boldsymbol{K}}))_{\mathbb{Z}_{(p)}}$.
We define the order of $F$ by
$$
\text{ord}_p(F):=\text{min}\{\,H\in\Lambda_2(\boldsymbol{K})\,\mid\,
a(F;H)\not\equiv 0 \pmod{p}\,\},
$$
where the ``minimum'' is defined in the sense of the above order. If
$F \equiv 0 \pmod{p}$, then we define $\text{ord}_p(F)=(\infty)$. In the case of
$\boldsymbol{K}=\mathbb{Q}(\sqrt{3}\,i)$, we can also define the order in a similar
way.

It is easy to check that
\begin{Lem}
\label{order}
The following equality holds.
$$
{\rm ord}_p(FG)={\rm ord}_p(F)+{\rm ord}_p(G).
$$
\end{Lem}
Then we have
\begin{Thm}
\label{sturm}
Let $k$ be an even integer and $p$ a prime number with $p\geq 5$.
Let $\boldsymbol{K}=\mathbb{Q}(i)$ or $\boldsymbol{K}=\mathbb{Q}(\sqrt{3}\,i)$.
For $F\in M_k(\Gamma^2(\mathcal{O}_{\boldsymbol{K}}),\nu_k)_{\mathbb{Z}_{(p)}}^{\text{sym}}$,
assume that
$$
{\rm ord}_p(F)\succ
\begin{cases}
\displaystyle
\left[ \left[\frac{k}{8}\right],2 \left[\frac{k}{8}\right], \left[\frac{k}{8}\right]\right] & 
\text{if $\boldsymbol{K}=\mathbb{Q}(i)$}
\vspace{2mm}
\\
\displaystyle
\left[\left[\frac{k}{9}\right],2\left[\frac{k}{9}\right], \left[\frac{k}{9}\right]\right] & 
\text{if $\boldsymbol{K}=\mathbb{Q}(\sqrt{3}\,i)$}
\end{cases},
$$
Then we have ${\rm ord}_p(F)=(\infty)$, i.e., $F \equiv 0\pmod{p}$.
Here
$$
\nu_k=
\begin{cases}
{\rm det}^{k/2} & \text{if $\boldsymbol{K}=\mathbb{Q}(i)$}\\
{\rm det}^{k} & \text{if $\boldsymbol{K}=\mathbb{Q}(\sqrt{3}\,i)$}
\end{cases}
$$
and $[x]$ inside of the bracket means the greatest integer such that $\leq x$.
\end{Thm}
\begin{Rem}
For 
$\boldsymbol{K}=\mathbb{Q}(i)$ (resp. $\boldsymbol{K}=\mathbb{Q}(\sqrt{3}\,i)$)
the matrix $\left[ [\frac{k}{8}],2[\frac{k}{8}],[\frac{k}{8}] \right]$\;
$\left(\text{resp.} \left[ [\frac{k}{9}],2[\frac{k}{9}],[\frac{k}{9}] \right]\right)$
is the maximum of the elements in $\Lambda_2(\boldsymbol{K})$ of the form
$\left[ [\frac{k}{8}],*,[\frac{k}{8}] \right]\geq 0$\;
(resp. $\left[ [\frac{k}{9}],*,[\frac{k}{9}] \right]\geq 0$).
\end{Rem}
\begin{proof} {\it of Theorem \ref{sturm}}
Let $\boldsymbol{K}=\mathbb{Q}(i)$. We use the induction
on the weight $k$. 
We can confirm that it is true for small $k$.
Suppose that the statement is true for any $k$ with $k<k_0$.
We shall prove that the statement is true for the weight $k_0$.

Let $\text{ord}_p(F)=[m_0,\alpha_0,n_0]\succ
\left[ [\frac{k}{8}],2 [\frac{k}{8}], [\frac{k}{8}]\right]$.
Applying Theorem \ref{structure} to 
$F$, we can write as
\begin{equation}
\label{expression}
F=P(E_{4,\boldsymbol{K}}^{(2)},E_{6,\boldsymbol{K}}^{(2)},F_{10},F_{12})+\chi_8\cdot G,
\end{equation}
where
$$
P\in\mathbb{Z}_{(p)}[x_1,x_2,x_3,x_4]\;\; \text{and}\;\;
G\in M_{k-8}(\Gamma^2(\mathcal{O}_{\boldsymbol{K}}),\nu_{k-8})_{\mathbb{Z}_{(p)}}^{\text{sym}}.
$$
Now we recall that
\begin{equation}
\label{restriction}
a(F\mid_{\mathbb{S}_2};[m,r,n])=
\sum_{\substack{a+bi\in\mathbb{Z}[i]\\ r=2a\\ 4mn-(a^2+b^2)\geq 0}}
a(F;[m,a+bi,n]).
\end{equation}
Restricting both sides of (\ref{expression}) to $\mathbb{S}_2$, we obtain
$$
F\mid_{\mathbb{S}_2}=P(G_4,G_6,6X_{10},X_{12}),
$$
where $G_k$ is the Siegel Eisenstein series of weight $k$ and $X_k$ is Igusa's
cusp form appeared in $\S$ \ref{Sect.2.2.2}. The identity (\ref{restriction}) implies
that
$$
\text{ord}_p(F\mid_{\mathbb{S}_2})\succ
\left[\left[\frac{k}{8}\right],2\left[\frac{k}{8}\right], \left[\frac{k}{8}\right]\right].
$$
In particular, we have
$$
\text{ord}_p(F\mid_{\mathbb{S}_2})\succ
\left[\left[\frac{k}{10}\right],r, \left[\frac{k}{10}\right]\right]
$$
for any $r\in\mathbb{Z}$. By Theorem 2.4 in Kikuta-Kodama-Nagaoka \cite{K-K-N}, we have
$F \equiv 0 \pmod{p}$. Therefore $P \equiv 0 \pmod{p}$ as a polynomial and
hence
$$
F \equiv \chi_8\cdot G \pmod{p}. 
$$
By Lemma \ref{order}, $\text{ord}_p(G)=[m_0-1,\alpha_0-(1+i),n_0-1]$ because
of $\text{ord}_p(\chi_8)=[1,1+i,1]$. It follows that
\begin{align*}
\text{ord}_p(G) & =[m_0-1,\alpha_0-(1+i),n_0-1]\\
                   & \succ \left[\left[\frac{k_0}{8}\right]-1,2\left[\frac{k_0}{8}\right]-(1+i), \left[\frac{k_0}{8}\right]-1\right]\\
                  & \succ \left[\left[\frac{k_0-8}{8}\right],2\left[\frac{k_0-8}{8}\right],
\left[\frac{k_0-8}{8}\right]\right].
\end{align*}
By the induction hypothesis, we have $G \equiv 0\pmod{p}$. This completes the proof
in the case $\boldsymbol{K}=\mathbb{Q}(i)$. \\
The proof in the case of
$\boldsymbol{K}=\mathbb{Q}(\sqrt{3}\,i)$ is almost the same as the case $\boldsymbol{K}=\mathbb{Q}(i)$.
\end{proof}
\begin{Cor}
\label{sturmcorollary}
Assume that $\boldsymbol{K}=\mathbb{Q}(i)$ (resp. $\mathbb{Q}(\sqrt{3}\,i)$) and
$p$ is a prime number with $p\geq 5$. If a Hermitian modular form
$F\in M_k(\Gamma^2(\mathcal{O}_{\boldsymbol{K}}),\nu_k)_{\mathbb{Z}_{(p)}}^{{\rm sym}}$
satisfies
$$
a(F;H) \equiv 0 \pmod{p}
$$
for all $H\in\Lambda_2(\boldsymbol{K})$ with 
${\rm tr}(H)\leq 2 \displaystyle \left[\frac{k}{8}\right]$ {\rm (}resp.
${\rm tr}(H)\leq 2 \displaystyle \left[\frac{k}{9}\right]${\rm )}, then
$$
F \equiv 0 \pmod{p}.
$$
\end{Cor}
\subsection{Theta operator}
\label{Sect.2.3}
We recall that the Fourier expansion of Hermitian modular form
can be regarded as an element of certain formal power series ring
$\mathbb{C}[\![\boldsymbol{q}]\!]$. The {\it theta operator}
over $\mathbb{C}[\![\boldsymbol{q}]\!]$ is defined as
$$
\varTheta : F=\sum a(F;H)\boldsymbol{q}^H \longmapsto
\varTheta (F):=\sum a(F;H)\cdot \text{det}(H)\boldsymbol{q}^H.
$$
In the case that $n=1$, the theta operator is equal to the Ramanujan
operator, which produces several interesting results (cf. Serre \cite{Serre}).
It should be noted that $\varTheta (F)$ is not necessarily a Hermitian
modular form even if $F$ is.

We fix a prime number $p$. If $F\in M_k(\Gamma,\nu_k)_{\mathbb{Z}_{(p)}}$
satisfies
$$
\varTheta (F) \equiv 0 \pmod{p},
$$
then we call $F$ an element of {\it the mod $p$ kernel of the theta operator}.
Assume that $F=\sum a(F;H)\boldsymbol{q}^H\in M_k(\Gamma,\nu_k)_{\mathbb{Z}_{(p)}}$.
If there is an integer $r$\,($r<n$) such that
$$
a(F;H) \equiv 0 \pmod{p}
$$
holds fro all $H\in\Lambda_n(\boldsymbol{K})$ with $\text{rank}(H)>r$, then
$F$ is called a {\it mod $p$ singular Hermitian modular form} (e.g. cf. B\"{o}cherer-Kikuta \cite{B-K}). 
It is obvious that,
if $F$ is a mod $p$ singular Hermitian modular form, then $F$ is an element of
mod $p$ kernel of the theta operator.

The main purpose of this paper is to give some examples of Hermitian modular
form in the mod $p$ kernel of the theta operator in the case that $n=2$.
\subsubsection{Basic property of theta operator}
\label{Sect.2.3.1}
As we stated above, the image $\varTheta (F)$ is not necessarily a Hermitian
modular form. However the following result holds:
\begin{Thm}
\label{thetamodp}
Assume that $\boldsymbol{K}=\mathbb{Q}(i)$ and $p$ is a prime number such
that $p\geq 5$. For any 
$F\in M_k(\Gamma^2(\mathcal{O}_{\boldsymbol{K}}),{\rm det}^{k/2})_{\mathbb{Z}_{(p)}}^{\text{sym}}$, there is a cusp form
$$
G\in S_{k+p+1}(\Gamma^2(\mathcal{O}_{\boldsymbol{K}}),{\rm det}^{(k+p+1)/2})_{\mathbb{Z}_{(p)}}^{{\rm sym}}
$$ 
such that
$$
\varTheta (F) \equiv G \pmod{p}.
$$
\end{Thm}
\begin{proof}
A corresponding statement in the case of Siegel modular forms can be found
in B\"{o}cherer-Nagaoka \cite{B-N}, Theorem 4. The proof here follows the same line.
We consider the normailized Rankin-Cohen bracket $[F_1,F_2]$ for
$F_i\in M_{k_i}(\Gamma^2(\mathcal{O}_{\boldsymbol{K}}),\text{det}^{k_i/2})_{\mathbb{Z}_{(p)}}^{\text{sym}}$\\
$(i=1,2)$
(e.g. cf. Martin-Senadheera \cite{M-S}). We can show that $[F_1,F_2]$ becomes
a Hermitian cusp form
$$
[F_1,F_2]\in S_{k_1+k_2+2}
(\Gamma^2(\mathcal{O}_{\boldsymbol{K}}),\text{det}^{(k_1+k_2+2)/2})
_{\mathbb{Z}_{(p)}}^{\text{sym}}.
$$
If $p$ is a prime number such that $p\geq 5$, then there is a Hermitian
modular form $G_{p-1}\in M_{p-1}(\Gamma^2(\mathcal{O}_{\boldsymbol{K}}),\text{det}^{(p-1)/2})_{\mathbb{Z}_{(p)}}^{\text{sym}}$ 
such that
$$
G_{p-1} \equiv 1 \pmod{p}\quad (\text{cf. {\rm Kikuta-Nagaoka} \cite{K-N}. Proposition 5}).
$$
For a given $F\in M_k(\Gamma^2(\mathcal{O}_{\boldsymbol{K}}),\text{det}^{k/2})_{\mathbb{Z}_{(p)}}^{\text{sym}}$, we have
$$
\varTheta (F) \equiv [F,G_{p-1}] \pmod{p}.
$$
Hence we may put
$$
G:=[F,G_{p-1}]\in  S_{k+p+1}(\Gamma^2(\mathcal{O}_{\boldsymbol{K}}),\text{det}^{(k+p+1)/2})_{\mathbb{Z}_{(p)}}^{\text{sym}}.
$$
\end{proof}
\begin{Ex}
\label{ex.1}
$$
\varTheta (E_{4,\boldsymbol{K}}) \equiv 5E_{4,\boldsymbol{K}}\cdot \chi_8+2F_{12}
\pmod{7}.
$$
\end{Ex}
\section{Eisenstein case}
\label{Sect.3}
In this section we deal with the Hermitian modular forms related to the Eisenstein series
of degree 2.
\subsection{Krieg's result}
\label{Sect.3.1}
We denote by $h_{\boldsymbol{K}}$ the class number of $\boldsymbol{K}$ and
$w_{\boldsymbol{K}}$ the order of the unit group of $\boldsymbol{K}$.

Given a prime $q$ dividing $D_{\boldsymbol{K}}:=-d_{\boldsymbol{K}}$ define the $q$-factor $\chi_q$ of 
$\chi_{\boldsymbol{K}}$ (cf. Miyake \cite{Miyake}, p.80). Then $\chi_{\boldsymbol{K}}$ can be
decomposed as
$$
\chi_{\boldsymbol{K}}=\prod_{q\mid D_{\boldsymbol{K}}}\chi_q.
$$
We set
$$
a_{D_{\boldsymbol{K}}}(\ell):=\prod_{q\mid D_{\boldsymbol{K}}}(1+\chi_q(-\ell)).
$$
Let $D_{\boldsymbol{K}}=mn$ with coprime $m$,\,$n$. We set
$$
\psi_m:=\prod_{\substack{q:\text{prime}\\ q\mid m}}\chi_q,\qquad
\psi_1:=1.
$$
For $H\in\Lambda_2(\boldsymbol{K})$ with $H\ne O_2$, we define
$$
\varepsilon (H):=\text{max}\{ \ell\in\mathbb{N}\,\mid\, \ell^{-1}H\in
\Lambda_2(\boldsymbol{K})\}.
$$
Krieg's result is stated as follows:
\begin{Thm}
\label{Krieg}
{\rm (Krieg \cite{Krieg})}\;
Assume that $k \equiv 0 \pmod{w_{\boldsymbol{K}}}$ and $k>4$.
Then there exists a modular form
$F_{k,\boldsymbol{K}}\in M_k(\Gamma_2(\mathcal{O}_{\boldsymbol{K}}))_{\mathbb{Q}}^{{\rm sym}}$
whose Fourier coefficient $a(F_{k,\boldsymbol{K}};H)$ is given by
$$
a(F_{k,\boldsymbol{K}};H)=
\begin{cases}
\displaystyle\frac{4k(k-1)}{B_k\cdot B_{k-1,\chi_{\boldsymbol{K}}}}
\sum_{0<d\mid\varepsilon (H)}d^{k-1}
G_{\boldsymbol{K}}\left(k-2;\frac{D_{\boldsymbol{K}}\cdot{\rm det}(H)}{d^2} \right)
{\rm if}\; H>0,\\
\displaystyle
-\frac{2k}{B_k}\sum_{0<d\mid \varepsilon (H)}d^{k-1}
\quad {\rm if\; rank}(H)=1,\\
1 \quad {\rm if}\; H=O_2,
\end{cases}
$$ 
where $B_m$(resp. $B_{m,\chi}$) is the Bernoulli (resp. the generalized Bernoulli)
number and
$$
G_{\boldsymbol{K}}(s;N):=
\frac{1}{a_{D_{\boldsymbol{K}}}(N)}
\sum_{0<d\mid N}
\sum_{\substack{mn=D_{\boldsymbol{K}}\\ (m,n)=1}}\psi_m(-N/d)\psi_n(d)d^s.
$$
\end{Thm}
\begin{Rem}
In \cite{Krieg}, Krieg stated that the modular form $F_{k,\boldsymbol{K}}$ coincides with the
weight $k$ Hermitian Eisenstein series (in his notation $E_2^k$) for any
$\boldsymbol{K}$. However it is known that it is true only for the case 
$h_{\boldsymbol{K}}=1$.
\end{Rem}

The first main result is as follows:
\begin{Thm}
\label{main1}
Let $F_{k,\boldsymbol{K}}$ be the Hermitian modular form introduced in Theorem \ref{Krieg}.
If $p>3$ is a prime number such that $\chi_{\boldsymbol{K}}(p)=-1$ and
$h_{\boldsymbol{K}} \not\equiv 0 \pmod{p}$, then
$$
\varTheta (F_{p+1,\boldsymbol{K}}) \equiv 0 \pmod{p}.
$$
\end{Thm}
\begin{Cor}
\label{cor1}
Assume that $h_{\boldsymbol{K}}=1$ and $p>3$ is a prime number such
that $\chi_{\boldsymbol{K}}(p)=-1$.
Then the weight $p+1$ Hermitian Eisenstein series $E_{p+1,\boldsymbol{K}}^{(2)}$ satisfies
$$
\varTheta (E_{p+1,\boldsymbol{K}}^{(2)}) \equiv 0 \pmod{p}.
$$
\end{Cor}
\begin{Rem}
{\rm (1)}\;
In the above theorem, the weight condition $k=p+1 \equiv 0 \pmod{w_{\boldsymbol{K}}}$
is automatically satisfied.
In fact, 
in the case $\boldsymbol{K}=\mathbb{Q}(i)$, the condition $\chi_{\boldsymbol{K}}(p)=-1$
implies $p \equiv 3 \pmod{4}$. Then $p+1 \equiv 0 \pmod{4}$. In the case 
$\boldsymbol{K}=\mathbb{Q}(\sqrt{3}\,i)$, it follows from 
the condition $\chi_{\boldsymbol{K}}(p)=-1$ that $p \equiv -1 \pmod{3}$. Since $p$ is odd,
we have $p+1 \equiv 0 \pmod{6}$. Since $w_{\boldsymbol{K}}=2$ in the other cases,
$p+1 \equiv 0 \pmod{w_{\boldsymbol{K}}}$ is obvious.
\\
{\rm (2)}\; It is known that there are infinitely many $\boldsymbol{K}$ and $p$ satisfying
$$
\chi_{\boldsymbol{K}}(p)=-1\;\; \text{and}\;\;
h_{\boldsymbol{K}} \not\equiv 0 \pmod{p}
$$
{\rm (e.g. cf. Horie-Onishi \cite{H-O})}.\\
{\rm (3)}\; Our interest is to construct an element of mod $p$ kernel of the theta operator with the possible minimum weight 
(i.e., the weight is its filtration. For the details on the filtration 
of the mod $p$ modular forms, see Serre \cite{Serre}). 
If we do not restrict on the weight, we can construct some trivial examples in several ways. 
For example, the power $F^p$ of a modular form $F$ is such a trivial example. 
If $F$ is of weight $k$, then its weight is $pk$ and this is too large.
We suppose that the possible minimum weight is $p+1$ for mod $p$ non-singular cases
(cf. B\"{o}cherer-Kikuta-Takemori \cite{B-K-T}).
\end{Rem}
For the proof of the theorem, it is sufficient to show the following:
\begin{Prop}
\label{prop1}
Assume that
$p>3$ is a prime number such that $\chi_{\boldsymbol{K}}(p)=-1$ and
$h_{\boldsymbol{K}} \not\equiv 0 \pmod{p}$. Let $a(F_{k,\boldsymbol{K}};H)$
be the Fourier coefficient of $F_{k,\boldsymbol{K}}$ at $H$. If ${\rm det}(H)\not\equiv 0\pmod{p}$,
then
\begin{equation}
\label{star}
a(F_{p+1,\boldsymbol{K}};H) \equiv 0 \pmod{p}.
\end{equation}
\end{Prop}
\begin{proof}
By Theorem \ref{Krieg}, the Fourier coefficient $a(F_{p+1,\boldsymbol{K}};H)$ is expressed as
$$
a(F_{p+1,\boldsymbol{K}};H)=\frac{4(p+1)p}{B_{p+1}\cdot B_{p,\chi_{\boldsymbol{K}}}}
\sum_{0<d\mid \varepsilon (H)}d^p\,G_{\boldsymbol{K}}\left(p-1;\frac{D_{\boldsymbol{K}}\cdot\text{det}(H)}{d^2} \right)
$$
for $H>0$. First we look at the factor
$$
A:=\frac{4(p+1)p}{B_{p+1}\cdot B_{p,\chi_{\boldsymbol{K}}}}.
$$
By Kummer's congruence relation, we obtain
\begin{align*}
& \displaystyle \bullet\quad\frac{B_{p+1}}{p+1} \equiv \frac{B_2}{2}=\frac{1}{12} \pmod{p}
\vspace{2mm}
\\
& \displaystyle \bullet\quad\frac{B_{p,\chi_{\boldsymbol{K}}}}{p}
\equiv (1-\chi_{\boldsymbol{K}}(p))B_{1,\chi_{\boldsymbol{K}}}                                                                         
 = (1-\chi_{\boldsymbol{K}}(p))\frac{-2h_{\boldsymbol{K}}}{w_{\boldsymbol{K}}}
                                                                              \pmod{p}.
\end{align*}
Since $p>3$, $\chi_{\boldsymbol{K}}(p)=-1$, and $h_{\boldsymbol{K}}\not\equiv 0\pmod{p}$, the factor $A$ is a $p$-adic unit.

Next we shall show that, if $\text{det}(H)\not\equiv 0\pmod{p}$, then the factor
$$
B:=\sum_{0<d\mid \varepsilon (H)}d^p\,G_{\boldsymbol{K}}\left(p-1;\frac{D_{\boldsymbol{K}}\cdot\text{det}(H)}{d^2} \right)
$$
satisfies 
$$
B \equiv 0 \pmod{p}.
$$
We note that
$$
\chi_{\boldsymbol{K}}\left(\frac{D_{\boldsymbol{K}}\cdot\text{det}(H)}{d^2} \right)=0\;\;\text{or}\;\; -1.
$$
Therefore, it is sufficient to show that, if $p\nmid N$, $\chi_{\boldsymbol{K}}(N)=0$ or $-1$, then
\begin{equation}
\label{starstar}
G_{\boldsymbol{K}}(p-1,N) \equiv 0 \pmod{p}.
\end{equation}
To prove the congruence relation (\ref{starstar}), we need a kind of product formula for $G_{\boldsymbol{K}}(s,N)$.
Let $S=S_{\boldsymbol{K}}$ be the set of prime number which ramifies in $\boldsymbol{K}$. For $N\in\mathbb{N}$,
we decompose $N$ as
$$
N=N_1\cdot N_2,\qquad N_1=\prod_{\substack{q\in S\\ q\mid N}}q^{\beta_q},\quad N_2=\prod_{\substack{q\notin S\\ q\mid N}}q^{\beta_q}.
$$
\begin{Lem}
\label{lemma2}
Notation is as above. We have
\begin{align*}
& \sum_{0<d\mid N}\sum_{\substack{mn=D_{\boldsymbol{K}}\\ (m,n)=1}}\psi_m(-N/d)\psi_n(d)d^s\\
&=\prod_{q\mid N_2}\sum_{t=0}^{\beta_q}\chi_{\boldsymbol{K}}(q)^tq^{st}\cdot
            \sum_{\substack{mn=D_{\boldsymbol{K}}\\ (m,n)=1}}\psi_m(-1)\psi_m(N_2)\cdot\prod_{\substack{q\mid m\\ q\mid N_1}}\psi_n(q)^{\beta_q}q^{s\beta_q}\\
&\quad \cdot \prod_{\substack{q\mid n\\ q\mid N_1}} \psi_m(q)^{\beta_q}.  
\end{align*}
\end{Lem}
\begin{proof} We have
\begin{align*}
& \sum_{0<d\mid N}\psi_m(-N/d)\psi_n(d)d^s
=\psi_m(-1)\prod_{q\mid N}\sum_{t=0}^{\beta_q}\psi_m(q^{\beta_q-t})\psi_n(q^t)q^{st}\\
& =\psi_m(-1)\prod_{q\mid N_2}\psi_m(q)^{\beta_q}\sum_{t=0}^{\beta_q}\chi_{\boldsymbol{K}}(q)^tq^{st}
\cdot\prod_{\substack{q\mid m\\ q\mid N_1}}\psi_n(q)^{\beta_q}q^{s\beta_q}
\cdot \prod_{\substack{q\mid n\\ q\mid N_1}} \psi_m(q)^{\beta_q}.
\end{align*}
Taking a summation over $m$ (and $n$), we obtain the desired formula. 
\end{proof}

We return to the proof of (\ref{starstar}). From the above lemma, we obtain
\begin{align*}
& G_{\boldsymbol{K}}(p-1,N)\\
& =\frac{1}{a_{D_{\boldsymbol{K}}}(N)}\prod_{q\mid N_2}\sum_{t=0}^{\beta_q}\chi_{\boldsymbol{K}}(q)^tq^{t(p-1)}
     \cdot \sum_{\substack{mn=D_{\boldsymbol{K}}\\ (m,n)=1}}\psi_m(-1)\psi_m(N_2)\\
& \cdot \prod_{\substack{q\mid m\\ q\mid N_1}}\psi_n(q)^{\beta_q}q^{(p-1)\beta_q}\cdot\prod_{\substack{q\mid n\\ q\mid N_1}}\psi_m(q)^{\beta_q}.
\end{align*}
Using this formula, we study the $p$-divisibility of $G(p-1,N)$ separately.
\vspace{2mm}
\\
(i)\; The case $\chi_{\boldsymbol{K}}(N)=\chi_{\boldsymbol{K}}(N_2)=-1.$\\
In this case, there is a prime number $q\mid N_2$
such that $q\notin S$, 
$\chi_{\boldsymbol{K}}(q)=-1$, and $\beta_q$ is odd.
For this prime number $q$, we have
$$
\sum_{t=0}^{\beta_q}\chi_{\boldsymbol{K}}(q)^tq^{(p-1)t} \equiv \sum_{t=0}^{\beta_q}(-1)^t=0 \pmod{p}.
$$
This implies that $G_{\boldsymbol{K}}(p-1,N) \equiv 0 \pmod{p}$ for this $N$.\\
(ii)\; The case $\chi_{\boldsymbol{K}}(N)=0$.\\
In this case, there is a prime number $q\in S$ with $q\mid N$. We may assume that $\chi_{\boldsymbol{K}}(N_2)=1$.
(If  $\chi_{\boldsymbol{K}}(N_2)=-1$, then the proof is reduced to the case (i).) In this case, we obtain
\begin{align*}
& \sum_{\substack{mn=D_{\boldsymbol{K}}\\ (m,n)=1}}\psi_m(-1)\psi_m(N_2)
   \prod_{\substack{q\mid m\\ q\mid N_1}}\psi_n(q)^{\beta_q}q^{(p-1)\beta_q}
\prod_{\substack{q\mid n\\ q\mid N_1}}\psi_m(q)^{\beta_q}\\
& \equiv
  \sum_{\substack{m\mid D_{\boldsymbol{K}}\\ (m,D_{\boldsymbol{K}}/m)=1}}\psi_m(-1)\psi_m(N_2)
   \prod_{\substack{q\mid m\\ q\mid N_1}}\psi_{D_{\boldsymbol{K}}/m}(q)^{\beta_q}\prod_{\substack{q\mid (D_{\boldsymbol{K}}/m)\\ q\mid N_1}}\psi_m(q)^{\beta_q} \\
& =\sum_{\substack{m\mid D_{\boldsymbol{K}}\\ (m,D_{\boldsymbol{K}}/m)=1}}
   \{ \psi_m(-1)\psi_m(N_2)+\psi_{D_{\boldsymbol{K}}/m}(-1)\psi_{D_{\boldsymbol{K}}/m}(N_2)\}\\
&\qquad \cdot \prod_{q\mid m}\psi_{D_{\boldsymbol{K}}/m}(q)^{\beta_q}\prod_{q\mid (D_{\boldsymbol{K}}/m)}\psi_m(q)^{\beta_q}\times\frac{1}{2} \pmod{p}.
\end{align*}
Since $\psi_m(-1)=-\psi_{D_{\boldsymbol{K}}/m}(-1)$ and $\psi _m(N_2)=
\psi_{D_{\boldsymbol{K}}/m}(N_2)$ (because $\chi_{\boldsymbol{K}}(N_2)=1$),
we obtain
$$
\sum_{\substack{m\mid D_{\boldsymbol{K}}\\ (m,D_{\boldsymbol{K}}/m)=1}}
 \{ \psi_m(-1)\psi_m(N_2)+\psi_{D_{\boldsymbol{K}}/m}(-1)\psi_{D_{\boldsymbol{K}}/m}(N_2)\}=0.
$$
This shows that $G_{\boldsymbol{K}}(p-1,N)\equiv 0 \pmod{p}$ again in this case.
We complete the proof of (\ref{starstar}).
\end{proof}
The second result of this section is related to the mod $p$ singular Hermitian modular forms.
\begin{Thm}
\label{modpsingular}
Assume that $p>3$ is a prime number such that
$p \equiv 3 \pmod{4}$ and $\boldsymbol{K}=\mathbb{Q}(\sqrt{p}\,i)$. Let $F_{k,\boldsymbol{K}}$
be the Hermitian modular form introduced in Theorem \ref{Krieg}. Then the modular form
$F_{\frac{p+1}{2},\boldsymbol{K}}$ is a mod $p$ singular Hermitian modular form.
\end{Thm}
\begin{proof}
From Theorem \ref{Krieg}, we have
$$
a(F_{\frac{p+1}{2},\boldsymbol{K}};H)
=\frac{(p+1)(p-1)}{B_{\frac{p+1}{2}}\cdot B_{\frac{p-1}{2},\chi_{\boldsymbol{K}}}}
\sum_{0<d\mid\varepsilon (H)}d^{\frac{p-1}{2}}
G_{\boldsymbol{K}}\left(\frac{p-3}{2},\frac{D_{\boldsymbol{K}}\cdot\text{det}(H)}{d^2}\right)
$$
for $H>0$. Since the factor of the summand on the right-hand side is rational integer,
it is sufficient to show
\begin{equation}
\label{bigstar}
\frac{(p+1)(p-1)}{B_{\frac{p+1}{2}}\cdot B_{\frac{p-1}{2},\chi_{\boldsymbol{K}}}}
                      \equiv 0 \pmod{p}.
\end{equation}
We have the following result:
\begin{Lem}
\label{lemma3}
Assume that 
$p>3$, $p \equiv 3 \pmod{4}$, and $\boldsymbol{K}=\mathbb{Q}(\sqrt{p}\,i)$. 
Then we have
$$
{\rm (i)}\;B_{\frac{p+1}{2}}\not\equiv 0\pmod{p}
\qquad\qquad
{\rm (ii)}\; p\cdot B_{\frac{p-1}{2},\chi_{\boldsymbol{K}}} \equiv -1 \pmod{p}.
$$
\end{Lem}
\begin{proof}
(i) The statement $B_{\frac{p+1}{2}}\not\equiv 0\pmod{p}$ can be found in, for example,
Washington \cite{Washington}, p.86, Exercise 5.4.\\
(ii)\; The congruence $p\cdot B_{\frac{p-1}{2},\chi_{\boldsymbol{K}}} \equiv -1 \pmod{p}$
is a special case of the theorem of von Staudt-Clausen for the generalized Bernoulli
numbers. For the proof see Carlitz \cite{Carlitz}, Theorem 3.
\end{proof}
We return to the proof of (\ref{bigstar}). From the above lemma, we have
$$
B_{\frac{p+1}{2}}\cdot B_{\frac{p-1}{2},\chi_{\boldsymbol{K}}}\in \frac{1}{p}\mathbb{Z}_{(p)}^\times. 
$$
This implies (\ref{bigstar}).
\end{proof}
\section{Theta series case}
\label{Sect.4}
In this section we construct Hermitian modular forms in the mod $p$ kernel of
theta operator defined from theta series.

For a positive Hermitian lattice $\mathcal{L}$ of rank $r$, we associate the
{\it Hermitian theta series}
$$
\vartheta_{\mathcal{L}}^{(n)}(Z)=\vartheta_H^{(n)}(Z)=
\sum_{X\in\mathcal{O}_{\boldsymbol{K}}^{(r,n)}}\text{exp}(\pi i\text{tr}({}^t\overline{X}HXZ)),
\quad Z\in\mathbb{H}_n,
$$
where $H$ is the corresponding Gram matrix of $\mathcal{L}$.

In the rest of this paper, we assume that
$$
\boldsymbol{K}=\mathbb{Q}(i).
$$
\subsection{Theta series for unimodular Hermitian lattice of rank 8}
\label{rank8}
We denote by $\mathcal{U}_r(\mathcal{O}_{\boldsymbol{K}})=\mathcal{U}_r(\mathbb{Z}[i])$
the set of even integral, positive definite unimodular Hermitian matrices of rank $r$ over
$\mathcal{O}_{\boldsymbol{K}}=\mathbb{Z}[i]$. It is known that $4\mid r$.We denote by 
$\widetilde{\mathcal{U}}_r(\mathcal{O}_{\boldsymbol{K}})$ the set of unimodular equivalence
classes. It is also known that 
$|\widetilde{\mathcal{U}}_8(\mathcal{O}_{\boldsymbol{K}})|=3$.
We fix a set of representatives $\{ H_1,H_2,H_3\}$, in which $H_i$ have the following
data:
$$
|\text{Aut}(H_1)|=2^{15}\cdot 3^5\cdot 5^2\cdot 7,\quad
|\text{Aut}(H_2)|=2^{22}\cdot 3^2\cdot 5\cdot 7,\quad
|\text{Aut}(H_3)|=2^{21}\cdot 3^4\cdot 5^2,
$$
(cf. http://www.math.uni-sb.de/ag/schulze/Hermitian-lattices/).

The following identity is a special case of Siegel's main formula for Hermitian
forms:
\begin{equation}
\label{mainformula}
\frac{\vartheta_{H_1}^{(2)}}{2^{15}\cdot 3^5\cdot 5^2\cdot 7}+
\frac{\vartheta_{H_2}^{(2)}}{2^{22}\cdot 3^2\cdot 5\cdot 7}+
\frac{\vartheta_{H_3}^{(2)}}{2^{21}\cdot 3^4\cdot 5^2}
=\frac{61}{2^{22}\cdot 3^5\cdot 5\cdot 7}E_{8,\boldsymbol{K}}^{(2)},
\end{equation}
where $\frac{61}{2^{22}\cdot 3^5\cdot 5\cdot 7}$ is the mass of the genus of the
unimodular Hermitian lattices in rank 8. The space
$M_8(\Gamma^2(\mathcal{O}_{\boldsymbol{K}}),\nu_8)^{\text{sym}}=
M_8(\Gamma^2(\mathcal{O}_{\boldsymbol{K}}))^{\text{sym}}$ is spanned by 
$(E_{4,\boldsymbol{K}}^{(2)})^2$ and $\chi_8$ (cf. Theorem \ref{structure}).
\begin{Lem}
\label{span}
We have the following identities
$$
\vartheta_{H_1}^{(2)}=(E_{4,\boldsymbol{K}}^{(2)})^2-5760\chi_8,\;\;
\vartheta_{H_2}^{(2)}=(E_{4,\boldsymbol{K}}^{(2)})^2-3072\chi_8,\;\;
\vartheta_{H_1}^{(2)}=(E_{4,\boldsymbol{K}}^{(2)})^2.
$$
\end{Lem}
\begin{proof}
The above identities come from the following data:
\begin{align*}
& \begin{cases}
   a(\vartheta_{H_1}^{(2)};[1,0,1])=120960,\\
   a(\vartheta_{H_1}^{(2)};[1,1+i,1])=0,
   \end{cases}
\begin{cases}
   a(\vartheta_{H_2}^{(2)};[1,0,1])=131712,\\
   a(\vartheta_{H_2}^{(2)};[1,1+i,1])=2688,
   \end{cases}
\\
& \begin{cases}
   a(\vartheta_{H_3}^{(2)};[1,0,1])=144000,\\
   a(\vartheta_{H_3}^{(2)};[1,1+i,1])=5760.
   \end{cases}
\end{align*}
The calculations of the Fourier coefficients were done by Till Dieckmann.
\end{proof}
\begin{Thm}
\label{mod7}
We have $\vartheta_{H_i}^{(2)}\in M_8(\Gamma^2(\mathcal{O}_{\boldsymbol{K}}))_{\mathbb{Z}}^{\text{sym}}$ ($i=1,2,3$) and
$$
\varTheta(\vartheta_{H_1}^{(2)}) \equiv \varTheta(\vartheta_{H_2}^{(2)})
\equiv 0 \pmod{7}.
$$
\end{Thm}
\begin{proof}
The first statement is a consequence of the unimodularity of $H_i$.
By (\ref{mainformula}), we see that
$$
4\vartheta_{H_1}^{(2)} \equiv 5\cdot 61E_{8,\boldsymbol{K}}^{(2)} \pmod{7}.
$$
Moreover by Lemma \ref{span}, we have
$$
\vartheta_{H_1}^{(2)} \equiv \vartheta_{H_2}^{(2)} \pmod{7}.
$$
Since $\varTheta(E_{8,\boldsymbol{K}}^{(2)}) \equiv 0 \pmod{7}$ (cf. Corollary \ref{cor1}),
we obtain
$$
\varTheta(\vartheta_{H_1}^{(2)}) \equiv \varTheta(\vartheta_{H_2}^{(2)})
\equiv 0 \pmod{7}.
$$
\end{proof}
\subsection{Theta series for unimodular Hermitian lattice of rank 12}
\label{HermitianLeech}
It is known that there is a unimodular Hermitian lattice 
$\mathcal{L}_{\mathbb{C}}\in\mathcal{U}_{12}(\mathbb{Z}[i])$
which does not have any vector of length one.
The transfer of this lattice to $\mathbb{Z}$ is the Leech lattice
(cf. http://www.math.uni-sb.de/ag/schulze/Hermitian-lattices/).
For this lattice, we have
$$
\vartheta_{\mathcal{L}_{\mathbb{C}}}^{(2)}\mid_{\mathbb{S}_2} =
\vartheta_{\text{Leech}}^{(2)},
$$
namely, the restriction of Hermitian theta series $\vartheta_{\mathcal{L}_{\mathbb{C}}}^{(2)}$
to the Siegel upper half-space coincides with the Siegel theta series
$\vartheta_{\text{Leech}}^{(2)}$ attached to the Leech lattice.

We fix the lattice
$\mathcal{L}_{\mathbb{C}}$ and call it here the {\it Hermitian Leech lattice}.
\begin{Thm}
\label{mod11cong}
Let $\mathcal{L}_{\mathbb{C}}$ be the Hermitian Leech lattice.
The attached Hermitian theta series 
$\vartheta_{\mathcal{L}_{\mathbb{C}}}^{(2)}
\in M_{12}(\Gamma^2(\mathcal{O}_{\boldsymbol{K}}))_{\mathbb{Z}}^{\text{sym}}$
satisfies the following congruence relations.
\vspace{2mm}
\\
{\rm (1)}\; $\varTheta(\vartheta_{\mathcal{L}_{\mathbb{C}}}^{(2)}) \equiv 0 \pmod{11}$.
\vspace{2mm}
\\
{\rm (2)}\; $\vartheta_{\mathcal{L}_{\mathbb{C}}}^{(2)} \equiv 1 \pmod{13}$.
\end{Thm}
\begin{proof}
(1)\, By Theorem \ref{thetamodp}, there is a Hermitian cusp form
$$
G\in S_{12+11+1}(\Gamma^2(\mathcal{O}_{\boldsymbol{K}})_{\mathbb{Z}_{(11)}}^{\text{sym}}
=S_{24}(\Gamma^2(\mathcal{O}_{\boldsymbol{K}})_{\mathbb{Z}_{(11)}}^{\text{sym}}
$$ 
such that
$$
\varTheta(\vartheta_{\mathcal{L}_{\mathbb{C}}}^{(2)}) \equiv G \pmod{11}.
$$
By Table 2, we see that
$$
a(\varTheta(\vartheta_{\mathcal{L}_{\mathbb{C}}}^{(2)});H) \equiv
a(G;H) \pmod{11}
$$
for any $H\in\Lambda_2(\boldsymbol{K})$ with $\text{rank}(H)=2$ and
$\text{tr}(H)\leq\displaystyle 2\left[\frac{24}{8}\right]=6$. Applying Sturm's
bound (Corollary \ref{sturmcorollary}), we obtain
$$
\varTheta(\vartheta_{\mathcal{L}_{\mathbb{C}}}^{(2)}) \equiv G \equiv 0 \pmod{11}.
$$
(2)\, We can confirm
$$
\Phi(\vartheta_{\mathcal{L}_{\mathbb{C}}}^{(2)})
=(E_4^{(1)})^3-720\Delta \equiv E_{12}^{(1)} \equiv 1 \pmod{13},
$$
where $\Phi$ is the Siegel operator and 
$\Delta=\frac{1}{1728}((E_4^{(1)})^3-(E_6^{(1)})^2)$ is Ramanujan's weight 12 cusp
form for $SL_2(\mathbb{Z})$. This shows that
$$
a(\varTheta(\vartheta_{\mathcal{L}_{\mathbb{C}}}^{(2)});H) \equiv
a(\varTheta(E_{12,\boldsymbol{K}}^{(2)});H) \pmod{13}
$$
for any $H\in\Lambda_2(\boldsymbol{K})$ with $\text{rank}(H)\leq 1$. Considering
this fact and Table 2, we see that this congruence relation holds for any
 $H\in\Lambda_2(\boldsymbol{K})$ with $\text{rank}(H)\leq 2$. Applying Sturm's
bound again, we obtain
$$
\vartheta_{\mathcal{L}_{\mathbb{C}}}^{(2)} \equiv E_{12,\boldsymbol{K}}^{(2)}
\equiv 1 \pmod{13}.
$$
\end{proof}
In this section,
we constructed Hermitian modular forms in the mod $p$ kernel of the theta
operator by theta series attached to unimodular Hermitian lattices. 
By the results above,
it is expected that, if $p$ is a prime number such that $p \equiv 3 \pmod{4}$,
then there is a unimodular lattice $\mathcal{L}$ of rank $p+1$ such that
$$
\varTheta(\vartheta_{\mathcal{L}}^{(2)}) \equiv 0 \pmod{p}.
$$
\subsection{Theta constants}
\label{thetaconstant}
In the previous sections, we gave examples of Hermitian modular form in the mod $p$
kernel of the theta operator. In this section we give another example.

The Hermitian
{\it theta constant} on $\mathbb{H}_2$ with characteristic $\boldsymbol{m}$
over $\boldsymbol{K}=\mathbb{Q}(i)$ is defined by
\begin{align*}
\theta_{\boldsymbol{m}}(Z) &=\theta(Z;\boldsymbol{a},\boldsymbol{b})\\
&:=\sum_{\boldsymbol{g}\in\mathbb{Z}[i]^{(2,1)}}\text{exp}
\left[
\frac{1}{2}\left(Z\left\{\boldsymbol{g}+\frac{1+i}{2}\boldsymbol{a}\right\}
+2\text{Re}\frac{1+i}{2}{}^t\boldsymbol{b}\boldsymbol{a}\right)
\right],
\quad Z\in \mathbb{H}_2,
\end{align*}
where $\boldsymbol{m}=\binom{\boldsymbol{a}}{\boldsymbol{b}}$, 
$\boldsymbol{a},\,\boldsymbol{b}\in\mathbb{Z}[i]^{(2,1)}$, $A\{ B\}={}^t\overline{B}AB$.
Denote by $\mathcal{E}$ the set of even characteristic of degree 2 mod 2
(cf. \cite{Freitag}), namely, 
$$
\mathcal{E}=
\left\{
\begin{pmatrix}0\\0\\0\\0\end{pmatrix},
\begin{pmatrix}0\\0\\0\\1\end{pmatrix},
\begin{pmatrix}0\\0\\1\\0\end{pmatrix},
\begin{pmatrix}0\\0\\1\\1\end{pmatrix},
\begin{pmatrix}0\\1\\0\\0\end{pmatrix},
\begin{pmatrix}0\\1\\1\\0\end{pmatrix},
\begin{pmatrix}1\\0\\0\\0\end{pmatrix},
\begin{pmatrix}1\\0\\0\\1\end{pmatrix},
\begin{pmatrix}1\\1\\0\\0\end{pmatrix},
\begin{pmatrix}1\\1\\1\\1\end{pmatrix}
\right\}
$$
\begin{Thm}
{\rm (Freitag \cite{Freitag})}\; Set
$$
\psi_{4k}(Z):=\frac{1}{4}\sum_{\boldsymbol{m}\in\mathcal{E}}\theta_{\boldsymbol{m}}^{4k}(Z),
\quad (k\in\mathbb{N}),
$$
Then
$$
\psi_{4k}\in M_{4k}(\Gamma^2(\mathcal{O}_{\boldsymbol{K}}))_{\mathbb{Z}}^{\text{sym}}.
$$
\end{Thm}
\begin{Rem}
(1)\, $\psi_4=E_{4,\boldsymbol{K}}^{(2)}$.\\
(2)\,$F_{10}=2^{-12}\displaystyle\prod_{\boldsymbol{m}\in\mathcal{E}}\theta_{\boldsymbol{m}}$,
where $F_{10}$ is a Hermitian modular form given in $\S$ \ref{Sect.2.2.2}.
\end{Rem}
\begin{Thm}
\label{thetaconstantth}
The following congruence relations holds.
\vspace{2mm}
\\
{\rm (1)}\; $\varTheta (\psi_8) \equiv 0\pmod{7}$.
\vspace{2mm}
\\
{\rm (2)}\; $\varTheta (\psi_{12}) \equiv 0 \pmod{11}$.
\end{Thm}
\begin{proof}
The form $\psi_8$ can be expressed as
$$
\psi_8=\frac{14}{75}(E_{4,\boldsymbol{K}}^{(2)})^2+\frac{61}{75}E_{8,\boldsymbol{K}}^{(2)}.
$$
Since $\varTheta(E_{8,\boldsymbol{K}}^{(2)}) \equiv 0 \pmod{7}$, we obtain
 $\varTheta (\psi_8) \equiv 0\pmod{7}$.\\
(2)\, We have the following expression.
\begin{align*}
& \vartheta_{\mathcal{L}_{\mathbb{C}}}^{(2)}
   =a_1\psi_{12}+a_2(E_{4,\boldsymbol{K}}^{(2)})^3
+a_3E_{4,\boldsymbol{K}}^{(2)}E_{8,\boldsymbol{K}}
+a_4E_{12,\boldsymbol{K}}^{(2)},\\
&
a_1=\frac{1470105}{8511808},\quad
a_2=\frac{167218051}{638385600},\quad
a_3=-\frac{147340193}{212795200},\\
&
a_4=\frac{802930253}{638385600},
\end{align*}
where $\mathcal{L}_{\mathbb{C}}$ is the Hermitian Leech lattice as before.
It should be noted that $a_2 \equiv a_3 \equiv 0\pmod{11}$, and $a_1\not\equiv 0\pmod{11}$.
Since 
$\varTheta (\vartheta_{\mathcal{L}_{\mathbb{C}}}^{(2)}) \equiv 
\varTheta(E_{12,\boldsymbol{K}}^{(2)}) \equiv 0\pmod{11}$, we obtain
$$
\varTheta (\psi_{12}) \equiv 0 \pmod{11}.
$$
\end{proof}
\section{Tables}
In this section, we summarize tables which are needed in the proof of our statements
in the previous sections.
\subsection{Theta series for rank 8 unimodular Hermitian lattices}
We introduced Hermitian theta series $\vartheta_{H_i}^{(2)}$ in $\S$ \ref{rank8}.
The following table gives some examples of Fouirer coefficients of  $\vartheta_{H_i}^{(2)}$
$(i=1,2)$ and $\vartheta_{[2,2,4]}^{(2)}$.
\begin{table*}[hbtp]
\caption{Fourier coefficients of hermitian theta series (rank 8)}
\label{tab.1}
\begin{center}
\begin{tabular}{llll} \hline
$H$       & 4det$(H)$  &  $a(\vartheta_{H_1}^{(2)};H)$   &  $a(\vartheta_{H_2}^{(2)};H)$     \\ \hline
$[0,0,0]$      & $0$           &         $1$                          &                  $1$ \\ \hline
$[1,0,0]$       & $0$           &         $480$                       &               $480$  \\ \hline
$[2,0,0]$       & $0$           &      $61920$                       &                $61920$  \\ \hline
$[3,0,0]$       & $0$           &  $1050240$                      &                  $1050240$ \\ \hline
$[4,0,0]$       & $0$           &  $7926240$                       &                $7926240$  \\ \hline
$[1,1+i,1]$     & $2$ &         $0$                          &        $2688=2^7\cdot 3\cdot 7$ \\ \hline
$[1,1,1]$       & $3$ &  $26880=2^8\cdot 3\cdot 5\cdot 7$&  $21504=2^{10}\cdot 3\cdot 7$       \\ \hline
$[2,2,1]$      & $4$           &  $120960=2^7\cdot 3^3\cdot 5\cdot 7$&  $131712=2^7\cdot 3\cdot 7^3$  \\ \hline
$[2,1+i,1]$    & $6$ & $1505280=2^{11}\cdot 3\cdot 5\cdot 7^2$   & $1483776=2^{10}\cdot 3^2\cdot 7\cdot 23$  \\ \hline
$[3,2+i,1]$     & $7$ & $3663360=2^9\cdot 3^3\cdot 5\cdot 53$   & $3717120=2^{11}\cdot 3\cdot 5\cdot 11^2$  \\ \hline
$[2,0,1]$      & $8$          &   $8346240=2^7\cdot 3^4\cdot 5\cdot 7\cdot 23$ & $8217216=2^7\cdot 3^2\cdot 7\cdot 1019$\\ \hline
$[2,2+2i,2]$    & $8$        &  $8346240=2^7\cdot 3^4\cdot 5\cdot 7\cdot 23$ & $8561280=2^7\cdot 3\cdot 5\cdot 7^3\cdot 13$\\ \hline
$[3,1+i,1]$     & $10$ & $30965760=2^{15}\cdot 3^3\cdot 5\cdot 7$  & $30992640=2^8\cdot 3\cdot 5\cdot 7\cdot 1153$  \\ \hline
$[2,2+i,2]$     & $11$& $55883520=2^8\cdot 3^4\cdot 5\cdot 7^2\cdot 11$ & $55716864=2^{10}\cdot 3\cdot 7\cdot 2591$ \\ \hline
$[4,2+i,1]$    & $11$&  $55883520=2^8\cdot 3^4\cdot 5\cdot 7^2\cdot 11$ & $55716864=2^{10}\cdot 3\cdot 7\cdot 2591$ \\ \hline
$[3,0,1]$       & $12$           &    $67751040=2^7\cdot 3\cdot 5\cdot 7\cdot 71^2$   &  $68353152=2^7\cdot 3\cdot 7\cdot 59\cdot 431$ \\ \hline
$[2,2,2]$      & $12$           & $96875520=2^{10}\cdot 3\cdot 5\cdot 7\cdot 17\cdot 53$& $96789504=2^{10}\cdot 3\cdot 7^2\cdot 643$ \\ \hline
$[2,1+i,2]$     & $14$ & $240537600=2^{12}\cdot 3^4\cdot 5^2\cdot 29$& $240752640=2^{11}\cdot 3\cdot 5\cdot 17\cdot 461$ \\ \hline
$[4,1+i,1]$    & $14$ &  $240537600=2^{12}\cdot 3^4\cdot 5^2\cdot 29$& $240752640=2^{11}\cdot 3\cdot 5\cdot 17\cdot 461$  \\ \hline
$[2,1,2]$      & $15$& $358095360=2^9\cdot 3\cdot 5\cdot 7\cdot 6661$ & $358041600=2^{11}\cdot 3^3\cdot 5^2\cdot 7\cdot 37$ \\ \hline
$[4,1,1]$       & $15$&$358095360=2^9\cdot 3\cdot 5\cdot 7\cdot 6661$ & $358041600=2^{11}\cdot 3^3\cdot 5^2\cdot 7\cdot 37$ \\ \hline
$[2,0,2]$      & $16$           & $544440960=2^7\cdot 3^3\cdot 5\cdot 7^2\cdot 643$& $544612992=2^7\cdot 3\cdot 7\cdot 11\cdot 113\cdot 163$ \\ \hline
$[4,0,1]$      & $16$           & $528958080=2^7\cdot 3^3\cdot 5\cdot 7\cdot 4373$ &  $527753856=2^7\cdot 3\cdot 7\cdot 196337$ \\ \hline                                  
\end{tabular}
\end{center}
\end{table*}
\newpage
\subsection{Theta series for rank 12 unimodular Hermitian lattice}
The table deals with the Fourier coefficients for theta series
$\vartheta_{\mathcal{L}_{\mathbb{C}}}^{(2)}$ where $\mathcal{L}_{\mathbb{C}}$
is the Hermitian Leech lattices introduced in $\S$ \ref{HermitianLeech}.
\begin{table*}[hbtp]
\caption{Non-zero Fourier coefficients $a(\vartheta_{\mathcal{L}}^{(2)};H)$ with rank$(H)=2$ and tr$(H)\leq 6$}
\label{tab.2}
\begin{center}
\begin{tabular}{lll} \hline
$H$       & 4det$(H)$  &  $a(\vartheta_{\mathcal{L}}^{(2)};H)$    \\ \hline
$[2,0,2]$    & $16$   &  $8484315840=2^6\cdot 3^5\cdot 5\cdot 7\cdot 11\cdot 13\cdot 109$       \\ \hline
$[2,1,2]$    & $15$   &  $4428103680=2^{15}\cdot 3^3\cdot 5\cdot 7\cdot 11\cdot 13$             \\ \hline
$[2,2,2]$    & $12$   &  $484323840=2^9\cdot 3^3\cdot 5\cdot 7^2\cdot 11\cdot 13$               \\ \hline
$[2,1+i,2]$  & $14$   &  $2214051840=2^{14}\cdot 3^3\cdot 5\cdot 7\cdot 11\cdot 13$             \\ \hline
$[2,2+i,2]$  & $11$   &  $201277440=2^{14}\cdot 3^3\cdot 5\cdot 7\cdot 13$                      \\ \hline
$[2,2+2i,2]$ & $8$    &  $8648640=2^6 \cdot 3^3\cdot 5\cdot 7\cdot 11\cdot 13$                  \\ \hline
$[2,0,3]$    & $24$   &  $480449249280=2^{14}\cdot 3^3\cdot 5\cdot 7^2\cdot 11\cdot 13\cdot 31$ \\ \hline
$[2,1,3]$    & $23$   &  $314395361280=2^{15}\cdot 3^3\cdot 5\cdot 7\cdot 11\cdot 13\cdot 71$   \\ \hline
$[2,2,3]$    & $20$   &  $77491814400=2^{14}\cdot 3^3\cdot 5^2\cdot 7^2\cdot 11\cdot 13$        \\ \hline
$[2,1+i,3]$  & $22$   &  $201679994880=2^{15}\cdot 3^4\cdot 5\cdot 7\cdot 13\cdot 167$          \\ \hline
$[2,2+i,3]$  & $19$   &  $46495088640=2^{14}\cdot 3^4\cdot 5\cdot 7^2\cdot 11\cdot 13$          \\ \hline
$[2,2+2i,3]$ & $16$   &  $8302694400=2^{12}\cdot 3^4\cdot 5^2\cdot 7\cdot 11\cdot 13$           \\ \hline
$[2,0,4]$    & $32$   &  $8567040081600=2^6 \cdot 3^3\cdot 5^2\cdot 7\cdot 11\cdot 13\cdot 19\cdot 10427$ \\ \hline
$[2,1,4]$    & $31$   &  $6230341877760=2^{15}\cdot 3^4\cdot 5\cdot 7^2\cdot 11\cdot 13\cdot 67$ \\ \hline
$[2,2,4]$    & $28$   &  $2254596664320=2^{10}\cdot 3^4\cdot 5\cdot 7\cdot 11\cdot 13\cdot 5431$ \\ \hline
$[2,4,4]$    & $16$   &  $8484315840=2^6 \cdot 3^5\cdot 5\cdot 7\cdot 11\cdot 13\cdot 109$       \\ \hline
$[2,1+i,4]$  & $30$   &  $4487883079680=2^{14}\cdot 3^3\cdot 5\cdot 7\cdot 11\cdot 13\cdot 2027$ \\ \hline
$[2,2+i,4]$  & $27$   &  $1565334650880=2^{14}\cdot 3^3\cdot 5\cdot 7^2\cdot 11\cdot 13\cdot 101$ \\ \hline
$[2,2+2i,4]$ & $24$   &  $482870868480=2^9 \cdot 3^3\cdot 5\cdot 7^2\cdot 11\cdot 13\cdot 997$    \\ \hline
$[3,0,3]$    & $36$   &  $27374536949760=2^{16}\cdot 3^3\cdot 5\cdot 7\cdot 11^2\cdot 13\cdot 281$ \\ \hline
$[3,1,3]$    & $35$   &  $20648247459840=2^{15}\cdot 3^3\cdot 5\cdot 7\cdot 11\cdot 13\cdot 4663$ \\ \hline
$[3,2,3]$    & $32$   &  $8431662919680=2^{12}\cdot 3^3\cdot 5\cdot 7\cdot 11\cdot 13\cdot 15233$ \\ \hline
$[3,3,3]$    & $27$   &  $1539504046080=2^{15}\cdot 3^2\cdot 5\cdot 7^2\cdot 11\cdot 13\cdot 149$ \\ \hline
$[3,1+i,3]$  & $34$   &  $15436369428480=2^{16}\cdot 3^4\cdot 5\cdot 7^2\cdot 11\cdot 13\cdot 83$ \\ \hline
$[3,2+i,3]$  & $31$   &  $6137351700480=2^{16}\cdot 3^5\cdot 5\cdot 7^2\cdot 11^2\cdot 13$        \\ \hline
$[3,3+i,3]$  & $26$   &  $1053888675840=2^{16}\cdot 3^3\cdot 5\cdot 7^2\cdot 11\cdot 13\cdot 17$  \\ \hline
$[3,2+2i,3]$ & $28$   &  $2218479943680=2^{15}\cdot 3^4\cdot 5\cdot 7\cdot 11\cdot 13\cdot 167$ \\ \hline
$[3,3+2i,3]$ & $23$   &  $309967257600=2^{16}\cdot 3^3\cdot 5^2\cdot 7^2\cdot 11\cdot 13$       \\ \hline
$[3,3+3i,3]$ & $18$   &  $26568622080=2^{16}\cdot 3^4\cdot 5\cdot 7\cdot 11\cdot 13$            \\ \hline
\end{tabular}
\end{center}
\end{table*}
Any non-zero Fourier coefficient $a(\vartheta_{\mathcal{L}_{\mathbb{C}}}^{(2)};H)$ 
with rank$(H)=2$ and tr$(H)\leq 6$
coincides with one in the above list.

The calculation is based on the following expression:
$$
\vartheta_{\mathcal{L}_{\mathbb{C}}}^{(2)}:=
\frac{7}{12}(E_{4,\boldsymbol{K}}^{(2)})^3
+\frac{5}{12}(E_{6,\boldsymbol{K}}^{(2)})^2
-10080E_{4,\boldsymbol{K}}^{(2)}\chi_8-60480F_{12}.
$$
\newpage

%




Toshiyuki Kikuta\\ 
Fukuoka Institute of Technology\\
Faculty of Information Engineering\\
Department of Information and Systems Engineering   \\
Fukuoka 811-0295, Fukuoka, Japan\\
Email: kikuta@fit.ac.jp
\\
\\
Shoyu Nagaoka\\
Kindai University\\
Department of Mathematics\\
Higashi-Osaka\\
Osaka 577-8520\\
Email:nagaoka@math.kindai.ac.jp
\end{document}